\documentclass[11pt,oneside]{amsart}
\usepackage{amsmath}
\usepackage{amssymb}
\usepackage{amscd}
\usepackage{amsthm}
\usepackage{amsxtra}
\usepackage{latexsym}
\usepackage{graphicx}
\usepackage{comment}
\usepackage{tikz,tikz-3dplot}

\usepackage[all]{xy}
\numberwithin{equation}{section}
\theoremstyle{plain}
\newtheorem{theorem}[equation]{Theorem}

\newtheorem{proposition}[equation]{Proposition}

\newtheorem{lemma}[equation]{Lemma}

\theoremstyle{remark}

\theoremstyle{definition}
\newtheorem{definition}[equation]{Definition}

\newtheorem{example}[equation]{Example}

\newcommand{\C}{\mathbb C}

\newcommand{\Q}{\mathbb Q}

\newcommand{\Z}{\mathbb Z}

\begin{document}

\title{Singularities of Cox Rings of Fano Varieties}
\author{Morgan V Brown}
\email{mvbrown@math.berkeley.edu}

\begin{abstract}
Let $X$ be a smooth complete Fano variety over $\mathbb{C}$.  We show that the Cox ring $\bigoplus_{L\in\text{Pic}(X)}H^0(X,\mathcal{O}_X(L))$ is Gorenstein with canonical singularities.
\end{abstract}
\maketitle
\section{Introduction}
The Cox ring (or total coordinate ring) of an algebraic variety is a natural generalization of the homogeneous coordinate ring of a projective variety.  It was studied by Cox \cite{MR1299003} in the case of toric varieties, who showed that the Cox ring of a smooth toric variety is always a polynomial ring.  Cox rings have played an important role in invariant theory, specifically counterexamples to Hilbert's 14th problem and are closely related to the study of symbolic Rees algebras.  More recently, there have been applications of Cox rings to birational geometry and number theory.
\begin{definition}\label{coxring}Let $D_1,\ldots D_r$ be Weil divisors which form a $\mathbb{Z}$ basis for the torsion free part of $\text{Cl}(X)$.  
Then a Cox ring of $X$ is given by 
\[
\text{Cox}(X; D_1,\ldots, D_r)=
\bigoplus_{\langle a_1,\ldots a_r\rangle\in\mathbb{Z}^r} H^0(X,\mathcal{O}_X(a_1D_1+\ldots a_rD_r))
\subset K(X)[t_1^{\pm},\ldots t_r^{\pm}].
\] That is, $\text{Cox}(X; D_1,\ldots, D_r)$ is generated by elements of the form $ft_1^{\alpha_1}\ldots t_r^{\alpha_r}$, where $f$ is a rational functions with at worst poles of order $\alpha_i$ along $D_i$, and no other poles.
\end{definition}

We will generally assume that $\text{Cl}(X)$ has no torsion; in this case the Cox ring is independent of the generators chosen, and we refer to it as just $\text{Cox}(X)$.
Like the homogeneous coordinate ring of a smooth projective variety, the Cox ring of a smooth projective variety $X$ may have singularities.  A pair $(X,\Delta)$ consisting of a variety and an effective $\Q$-Divisor $\Delta$ is said to be log Fano if the pair is klt and $-(K_X+\Delta)$ is ample; such a pair is called log Calabi-Yau if the pair is lc and $K_X+\Delta$ is trivial.  The main theorem of this paper states that the Cox ring of a log Fano variety has log terminal singularities.  

\begin{theorem}\label{main}

Let $(X,\Delta)$ be a $\Q$-factorial log Fano pair over $\mathbb{C}$, and let $D_1,\ldots D_r$ be a basis for the torsion free part of $\text{Cl}(X)$. 

\begin{enumerate}

\item The ring $\text{Cox}(X;D_1\ldots D_r)$ is normal with log terminal singularities, and in particular is Cohen-Macaulay.

\item If $X$ is a smooth complete strict Fano variety, then $\text{Cox}(X)$ is Gorenstein with canonical singularities.
\end{enumerate}
\end{theorem}

Recent work of Hashimoto and Kurano \cite{1007.3327} computes the canonical modules of Cox rings.  They show that when $X$ is a normal variety whose class group is a finitely generated free abelian group and $\text{Cox}(X)$ is Noetherian, the canonical module of $\text{Cox}(X)$ is a rank one free module.  On a smooth Fano variety, the Picard group and hence the class group is a finitely generated free abelian group \cite[Prop 2.1.2]{MR1668575}.  Also log terminal singularities on a Gorenstein variety are canonical.  Hence the second statement of Theorem \ref{main} follows from the first.

Elizondo, Kurano, and Watanabe \cite{MR2058459} showed in all characteristics that the Cox ring of a normal variety with finitely generated torsion free class group is a unique factorization domain (UFD).  Popov \cite{0402154} has established in characteristic $0$ that Cox rings of smooth del Pezzo surfaces have rational singularities, and are hence Cohen-Macaulay.  Castravet and Tevelev \cite{MR2278756}  provided similar proofs for blowups of higher dimensional projective spaces at general points.

An earlier version of this paper conjectured that finitely generated Cox rings are generally Cohen-Macaulay.  I am very grateful to Yoshinori Gongyo \cite{goprivate} for pointing out a counterexample:  A very general algebraic hyper-K{\"a}hler $4$-fold will have Picard number $1$ \cite{MR2746467}, and hence have finitely generated Cox ring.  But if $X$ is hyper-K{\"a}hler, then $H^2(X,\mathcal{O}_X)$ is nonzero, so this ring cannot be Cohen-Macaulay.

In the earlier version, \ref{main} was stated only for the case $\Delta=0$.  The arguments work equally well for the log Fano case, and this is a more natural setting for many of the constructions.  Independently, Gongyo, Okawa, Sannai, and Takagi \cite{1201.1133} proved the log Fano case using reduction to positive characteristic.  They also proved a converse to \ref{main} which relates it to work of Schwede and Smith on global F-regularity \cite{MR2628797} and investigated the case of log Calabi-Yau Mori Dream Spaces.

\begin{theorem}\label{maincy}
Let $(X,\Delta)$ be a projective $\Q$-factorial log Calabi-Yau pair over $\mathbb{C}$ such that $X$ is a Mori Dream Space, and let $D_1,\ldots D_r$ be a basis for the torsion free part of $\text{Cl}(X)$.  Then the ring $\text{Cox}(X;D_1\ldots D_r)$ is normal with log canonical singularities.
\end{theorem}

Gongyo, Okawa, Sannai, and Takagi \cite{1201.1133} have shown Thm \ref{maincy} for case where $X$ is a surface, and more recently, Kawamata and Okawa \cite{kopreprint} proved this in arbitrary dimension along with its converse.

We use ideas from the minimal model program as well as those of Hu and Keel \cite{MR1786494} relating MMP, Cox rings, and GIT.

Our proof proceeds by induction on the Picard rank $\rho$ of $X$.  The case of Fano varieties with $\rho=1$ follows from general work of Tomari and Watanabe on normal $\mathbb{Z}$-graded rings \cite[Thm 2.6]{MR1866376}, although for completeness we provide a proof up to cyclic covering using different techniques in section \ref{rho1}.

For $\rho>1$ we may recover $X$ from $\text{Cox}(X)$ by means of a GIT quotient by a torus $\mathbb{G}_m^{\rho_X}$.  We will filter this quotient into a series of $\mathbb{G}_m$ quotients, and show inductively that the singularities never get too bad.  This requires constructing a variety $X'$ from $X$ which is a small compactification of the total space of a nontrivial $\mathbb{G}_m$ bundle on $X$. We must show that with a suitable choice of bundle, $X'$ is also Fano with $\Q$-factorial and log terminal singularities.

In section \ref{mds}, we review facts about Mori Dream Spaces, and show that if $X$ is a Mori Dream Space and $L$ a line bundle on $X$, then the projectivized vector bundle $Y=\mathbb{P}_X(\mathcal{O}_X\oplus\mathcal{O}_X(L))$ is also a Mori Dream Space.  We also show that the Cox ring of $X$ is a cyclic cover of the Cox ring of $X'$.  The goal of section \ref{sing} is to make birational modifications to $Y$ until we arrive at a normal $\Q$-factorial log Fano pair $(X',\Delta)$ which is a small compactification of the chosen $\mathbb{G}_m$ bundle.  Standard techiniques of MMP show that $(X',\Delta)$ is klt.  In section \ref{CY} the same is done for the log Calabi-Yau case.

Finally, in section \ref{final} we address the cyclic covers introduced in sections \ref{rho1} and \ref{mds} to complete the argument.

I would like to thank my advisor David Eisenbud, as well as Yoshinori Gongyo, Sean Keel, James McKernan, and Kevin Tucker for helpful comments, discussions, and suggestions.  I would also like to thank Shinnosuke Okawa for sending me the preprints \cite{1201.1133,kopreprint}.

\section{Fano varieties with $\rho=1$}\label{rho1}

When the Picard number of $X$ is just one, the Cox ring is the ring of sections of multiples of a divisor.  This makes it easier to study than the multigraded case, and questions of whether singularities of such rings are Cohen-Macaulay, Gorenstein, or rational were studied in detail by Watanabe \cite{MR632654}.  
\begin{theorem}\label{basecase}
\begin{enumerate}
\item Let $(X,\Delta)$ be a $\Q$-factorial log Fano pair such that $\rho_X=1$, and let $L$ be the ample generator of $\text{Pic}(X)$.  Then the homogeneous coordinate ring $\bigoplus H^0(X,\mathcal{O}(nL))$ has log terminal singularities.
\item Let $(X,\Delta)$ be a projective $\Q$-factorial log Calabi-Yau pair such that $\rho_X=1$, and let $L$ be the ample generator of $\text{Pic}(X)$.  Then the homogeneous coordinate ring $\bigoplus H^0(X,\mathcal{O}(nL))$ has log canonical singularities.
\end{enumerate}
\end{theorem}

When $X$ is strictly Fano, \ref{basecase} is a special case of a result of Tomari and Watanabe \cite[Thm 2.6]{MR1866376}.
\begin{proof}
Since $X$ is $\Q$-factorial and $\Delta$ is effective the pair $(X,0)$ does not have worse singularities than $(X,\Delta)$.  So we assume $\Delta=0$.  Let $Z=\text{Spec}\bigoplus H^0(X,\mathcal{O}(nL))$.  We must show that $Z$ is normal, $\Q$-factorial, and then compute discrepancies of a resolution of $Z$.  The variety $Z$ is normal since the ring $\bigoplus H^0(X,\mathcal{O}(nL))$ is integrally closed.  We will resolve the singularities of $Z$ in two steps. 

Let $\psi: \tilde{X}\to X$ be a resolution of singularities of $X$.    Now, let $S(L)$ be the symmetric algebra on $\mathcal{O}(L)$, and let $\mathbb{A}_X(L)=\textbf{Spec}_X(S(L))$.  This is the total space of the line bundle $L$, and has a projection map $\pi:\mathbb{A}_X(L)\to X$.  We will likewise define  $\mathbb{A}_{X'}(\psi^*L)=\textbf{Spec}_{X'}(S(\psi^*L))$.  There is a birational maps $\psi':\mathbb{A}_{X'}(\psi^*L)\to\mathbb{A}_X(L)$ which is a resolution of $\mathbb{A}_X(L)$.  There is a second birational map $f:\mathbb{A}_X(L)\to Z$ which contracts the zero section $E$ to a point.  Note that $E|_E=-L$.

Together, these two maps resolve the singularities of $Z$.  Note that $Z$ is $\Q$-factorial since the relative Picard and Weil class groups of $f$ both have rank $1$.  It remains only to check discrepancies.

The $\mathbb{A}^1$ bundle $\mathbb{A}_X(L)$ is smooth in codimension $2$ so by adjunction, $K_{\mathbb{A}_X(L)}=\pi^*K_X-E$.  On $Z$, a multiple of $K_Z$ is trivial at the cone point, so $f^*(K_Z)=\pi^*K_X-mE$, where $m$ is the nonnegative rational number satisfying $mL=-K_X$.  In the log Fano case, $m$ is strictly positive.  Thus $K_{\mathbb{A}_X(L)}=f^*K_Z+(m-1)E$.

Since none of the centers of blowups of $\psi'$ are contained in $E$, $\psi'^*E$ is the strict transform of $E$.  Thus in the log Fano case, the discrepancy of $E$ is $m-1>-1$, and in the log Calabi-Yau case the discrepancy $m-1\geq -1$.  The other discrepancies are the same as those of $\psi$, which are large enough by hypothesis.
\end{proof}
\begin{example}
Let $X$ be a hypersurface of degree $d>n$ in $\mathbb{P}^n$, where $n\geq 4$.  By the Lefschetz hyperplane theorem, $\text{Pic}(X)=\mathbb{Z}$, generated by $\mathcal{O}(1)$.  For all $m\in \mathbb{Z}$, $H^1(X,\mathcal{O}(m))=0$, so the ring $\text{Cox}(X)=\bigoplus H^0(X,\mathcal{O}(m))=k[x_0\ldots x_n]/f$, which is finitely generated.  Thus $X$ is a Mori Dream Space.  Note that our choice $d>n$ guarantees that $X$ is not a Fano variety.  However, $H^{n-1}(X,\mathcal{O}_X)\neq 0$, so by \cite[Theorem 1]{MR870733} the ring $\text{Cox}(X)$ does not have rational singularities.  This ring is however Cohen-Macaulay.
\end{example}
\section{Mori Dream Spaces}\label{mds}

Mori Dream Spaces were first introduced by Hu and Keel\cite{MR1786494}, and are so called because it is relatively easy to carry out the operations of the Mori Program on such a space.  Let $X$ be a projective variety, and $R$ a Cox ring for $X$.  The variety $X$ is called a Mori Dream Space if $X$ is $\Q$-factorial, $\text{Pic}_{\Q}(X)=N^1(X)$, and $R$ is finitely generated as a $\C$ algebra.  This is a very special property, and has many nice consequences for the birational geometry of $X$.  In particular, the nef and psuedoeffective cones of $X$ are both rational polyhedral cones, and every nef divisor on $X$ is semiample.

Note that Hu and Keel's definition of the Cox ring of $X$ \cite{MR1786494} requires the divisors $D_i$ to be Cartier. When $X$ is $\mathbb{Q}$-factorial their definition differs from ours by only a finite extension.  Thus it doesn't matter which ring we consider for questions of finite generation.  Also, when $X$ is smooth and $\text{Pic}(X)\cong \mathbb{Z}^r$, the definitions coincide.
\begin{theorem}\cite{MR2601039}
If $(X,\Delta)$ is a $\Q$-factorial log Fano pair then $X$ is a Mori Dream Space.
\end{theorem}

Thus log Fano varieties give a nice class of Mori Dream Spaces to study, though these are not the only examples of Mori Dream Spaces.  Given a Mori Dream Space $X$, we want an appropriate compactification $X'$ of a $\mathbb{G}_m$ bundle on $X$.  One obvious way to get a compactification is to take a projectivized vector bundle.  Given a vector bundle $\mathcal{E}$ on $X$, let $\mathbb{P}_X(\mathcal{E})=\textbf{Proj}_X(\oplus\text{Sym}^n(\mathcal{E}))$.  We will be solely concerned with the case where $\mathcal{E}=\mathcal{O}_X\oplus\mathcal{O}_X(L)$ where $L$ is a Cartier divisor on $X$.  In this case $\mathbb{P}_X(\mathcal{O}_X\oplus\mathcal{O}_X(L))$ is a compactification of the $\mathbb{G}_m$ bundle on $X$ associated to $L$.  There are two irreducible boundary divisors, corresponding to the $0$ section of $L$ and the section at $\infty$.  To better understand the birational geometry of $Y=\mathbb{P}_X(\mathcal{O}_X\oplus\mathcal{O}_X(L))$ we calculate the Cox ring:

\begin{theorem}\label{MDS}  Let $X$ be a Mori Dream Space.  Choose $D_1,\ldots D_r$ Weil divisors generating the torsion free part of $\text{Cl}(X))$, and let $L$ be a nontrivial Cartier divisor which is in the subgroup generated by the $D_i$.  Let $Y=\mathbb{P}_X(\mathcal{O}_X(L)\oplus\mathcal{O}_X)$, which is a $\mathbb{P}^1$ bundle over $X$ with projection $\pi: Y\to X$.  Then
\begin{enumerate}
\item
The divisors $\pi^*D_i, E_\infty$, where $E_\infty$ is the section of $X$ at infinity form a $\Z$-basis for the torsion free part of $\text{Cl}(Y)$.
\item
$\text{Cox}(Y;\pi^*D_1,\ldots,\pi^*D_r,E_\infty)\cong \text{Cox}(X;D_1,\ldots,D_r)[s,t]$
\item
$Y$ is also a Mori Dream Space.
\end{enumerate}
\end{theorem}
\begin{proof}
For the first statement, $Y\setminus E_\infty$ is an $\mathbb{A}^1$ bundle over $X$, so its class group is isomorphic to that of $X$, and is generated by the pullbacks of generators for $\text{Cl}(X)$.  We have an exact sequence \cite[II, Prop 6.5]{MR0463157}:
\[
\mathbb{Z}\to \text{Cl}(Y)\to \text{Cl}(X)\to 0
\]

It remains to show that the first map is injective.  This will follow since $L$ was nontorsion:  The restriction of $E_\infty$ to itself is it's normal bundle, which is $-L$.  This is nontorsion, so $E_\infty$ is not torsion in $Y$ either.

For the second statement, set $R=\text{Cox}(X;D_1,\ldots,D_r)$, $S=\text{Cox}(Y;\pi^*D_1,\ldots,\pi^*D_r,E_\infty)$.
Now, since $\mathbb{P}(\mathcal{O}(L)\oplus\mathcal{O})\cong \mathbb{P}(\mathcal{O}(-L)\oplus\mathcal{O})$ with the only difference being that the tautological invertible sheaf $\mathcal{O}(1)$ is twisted by $\pi^*(-L)$, we have that the zero section $E_0$ and the infinity section $E_\infty$ are related by $E_0\sim L+E_\infty$.  Let $y$ be the rational function with a zero of order one at $E_0$ and poles along $L$ and $E_\infty$.  

Choose an affine open set $U$ in $X$ such that $L$ is trivial on $U$, and $y$ has no poles in $U$.  Then $\mathbb{A}^1\times U$ is an affine open set of $Y$, with coordinate ring $\Gamma(U)[y]$. Thus the function field $K(Y)=K(X)(y)$.

We are ready to define the new variables $s$ and $t$ in $K(Y)[t_1^\pm,\ldots t_{r+1}^\pm]$.  The variable $t$ is defined to be $t_{r+1}$.  This is an element of $S$ since the rational function $1$ has no zeros or poles.  

Next we define the element $s$.  Let $\alpha_i$ be the coefficient of $D_i$ in $L$ for $1\leq i\leq r$, and let $\alpha_{r+1}=1$.  Then set $s=y\prod t_i^{\alpha_i}$, which is in $S$ since $y$ has poles only along $L$ and $E_\infty$.
$y$ and $t_{r+1}$ are algebraically independent over $R$, so $R[s,t]\subset S$.  It remains to be shown that every element of $S$ is in $R[s,t]$.  It suffices to check homogeneous elements under the $\mathbb{Z}^{r+1}$ grading, so let $\lambda=\frac{f(U,y)}{g(U,y)}$ be a rational function such that $\lambda\prod t_i^{\beta_i}$ is an element of $S$, where $f$ and $g$ are functions in the coordinate ring of $\mathbb{A}^1\times U$.

The rational function $\lambda$ has no poles along $E_0$, so if $g(U,y)$ is divisible by $y^n$ so is $f(x,y)$.  Assume therefore that $g(U,y)$ is not divisible by $y$.  Likewise, $\lambda$ cannot have any poles along any horizontal divisor, so we may assume $g$ is actually the pullback of a function in $\Gamma(U)$.  Now if $f(U,y)$ is divisible by $y^m$, we may divide $\lambda\prod t_i^{\beta_i}$ by $s^m$ and still have an element of $S$.  Thus we assume the lowest $y$-degree term of $\lambda$ is constant in $y$.  Note that $\frac{1}{s^m}\lambda \prod t_i^{\beta_i}$ has nonnegative degree in $t_{r+1}$, since the rational function has no zeros at $E_\infty$.  Now, the constant term of $\frac{\lambda}{y^m}$ is the restriction to this function to $E_0\cong X$, and this rational function has zeros and poles of the prescribed orders along the $D_i$, so the constant term belongs to $R[u,v]$.  If the only term was the constant term, we are finished, otherwise we may proceed by induction on the number of terms.

For the third statement we know that $Y$ has a finitely generated Cox ring.  Since $Y$ has $\Q$-factorial singularities if $X$ does, and every numerically trivial divisor on $Y$ is torsion, we conclude that $Y$ is also a Mori Dream Space.  
 
\end{proof}
Warning!  For an arbitrary vector bundle $\mathcal{E}$ on a Mori Dream Space $X$, $\mathbb{P}_X(\mathcal{E})$ might not be a Mori Dream Space!  In fact, this may fail even if $X$ is a toric variety\cite{1009.5238}.

Our goal is, given $X$, build a variety $X'$ with a similar Cox ring and a lower Picard number.  If $X$ is log Fano, then $X'$ should be also, likewise if $X$ is log Calabi-Yau $X'$ should be too.  Thus we may inductively reduce to the $\rho=1$ case.  The variety $Y=\mathbb{P}_X(\mathcal{O}_X\oplus\mathcal{O}_X(L))$ constructed is not appropriate for our purposes since the Picard number has gone up.  However for the right choice of $L$, a suitable birational modification of $Y$ will produce such an $X'$.

Let $Y_0\subset Y$ be the open set consisting of the complement of the boundary divisors $E_0$ and $E_\infty$.  Then $Y_0$ is a $\mathbb{G}_m$ bundle over $X$.  We will choose $X'$ to be a particular small compactification of $Y_0$; that is the boundary will have codimension at least $2$ in $X'$.  Thus we can canonically identify $\text{Cl}(X')\cong\text{Cl}(Y_0)$.

\begin{theorem}\label{cover}
Take $X$, $Y$, $Y_0$ as above.  Let $D_1 \ldots D_r$ be a basis for the torsion free part of $\text{Cl}(X)$, and assume $L=mD_r$ for some $m$.  Fix $\mu_m$ a primitive $m$th root of unity.  Let $\pi_0:Y_0\to X$ be the projection morphism.  Let $X'$ be a projective small compactification of $Y_0$.  

Then $\pi_0^*D_1\ldots \pi_0^*D_{r-1}$ generate the torsion free part of $\text{Cl}(X')$,
and $\text{Cox}(X';\pi_0^*D_1,\ldots \pi_0^*D_{r-1})$ is isomorphic to the $\mathbb{Z}/m$ invariant part of $\text{Cox}(X;D_1,\ldots D_r)$ under the action induced by $t_r^i\to \mu_m^it_r^i$.
\end{theorem}
\begin{proof}
By assumption, $L$ is nontorsion.  $\text{Cl}(Y_0)\cong \text{Cl}(X')$ is given by the quotient of $\text{Cl}(Y)$ by the subgroup generated by $E_0$ and $E_\infty$, and the only classes in $\pi_0^*\text{Cl}(X)$ of this form are multiples of $L$.  This proves the first statement.

For the second statement we must write down a map
\[
\alpha: R=\text{Cox}(X';\pi_0^*D_1,\ldots \pi_0^*D_{r-1})\to \text{Cox}(X;D_1,\ldots D_r)
\]
Note that because of the grading we need only define $\alpha$ on homogeneous elements of $R$.  As before, $K(X',y)=K(X)$, where $y$ has a zero of order $m$ along $\pi_0^*D_r$ and no other poles or zeros.  Choose $U\subset X$ so that rational functions are the ratios of regular functions on $U$.  

A homogeneous element of $R$ has the form $\frac{f(U,y)}{g(U,y)}\prod_{i<r}t_i^{a_i}$.  Since $\frac{f(U,y)}{g(U,y)}$ has no poles along horizontal divisors in $Y_0$, it is actually a Laurent polynomial in $y$.  So such an element of $R$ really has the form $\sum \lambda_i y^i \prod_{i<r}t_i^{a_i}$ where $\lambda_i\in K(X)$ are rational functions with poles of appropriate orders along $\pi_0^*D_1, \ldots \pi_0^*D_{r-1}$ and a pole of order at most $mi$ along $\pi_0^*D_r$.  Thus we get a well defined map $\alpha$ by sending $\sum \lambda_i y^i \prod_{i<r}t_i^{a_i}$ to $\sum \lambda_i s_r^{mi} \prod_{i<r}s_i^{a_i}$

This map $\alpha$ certainly injective since $t_r$ is algebraically independent of the other variables.  Also, given a homogeneous element $\lambda \prod_{i\leq r}s_i^{a_i}$ of $\text{Cox}(X;D_1,\ldots D_r)$ which is $\mathbb{Z}/m$ invariant, it is the image of the element $\lambda y^{a_i/m}\prod_{i<r}s_i^{a_i}$ of $r$.  So the image of the map is the $\mathbb{Z}/m$ invariant elements of $\text{Cox}(X;D_1,\ldots D_r)$.
\end{proof}

Before moving on it is instructive to consider a simple example.
\begin{example}
Let $X=\mathbb{P}^1\times\mathbb{P}^1$, and take $L$ to be the line bundle $(1,-1)$.  Then $X$ is a toric Fano variety, and $\text{Cox}(X)=k[a_0,a_1,b_0,b_1]$, where the $a_i$ have grading $(1,0)$ and the $b_i$ have grading $(0,1)$.  In our construction the variety $Y= \mathbb{P}(\mathcal{O}_X\oplus \mathcal{O}_X(1,-1))$.

$Y$ is also toric, and both boundary divisors are isomorphic to $X=\mathbb{P}^1\times\mathbb{P}^1$.  There is a contracting morphism which contracts each boundary divisor along a ruling down to a $\mathbb{P}^1$, and the image of this morphism is $\mathbb{P}^3$.  This is another toric Fano variety, whose Cox ring is again the polynomial ring in $4$ variables, this time graded by total degree.  In this case there is no cyclic cover, since $(1,-1)$ can be extended to a $\mathbb{Z}$ basis for the class group of $X$ and so $m=1$.

The variety $X'=\mathbb{P}^3$ is our expected small compactification of the $\mathbb{G}_m$ bundle on $X$ corresponding to $L$.  In fact the map from the open locus can be seen geometrically since projection from each line in $\mathbb{P}^3$ gives a map from the complement of that line to a $\mathbb{P}^1$.  Since the example is toric, we exhibit this contraction in terms of polytopes in Figure \ref{p1crossp1}.

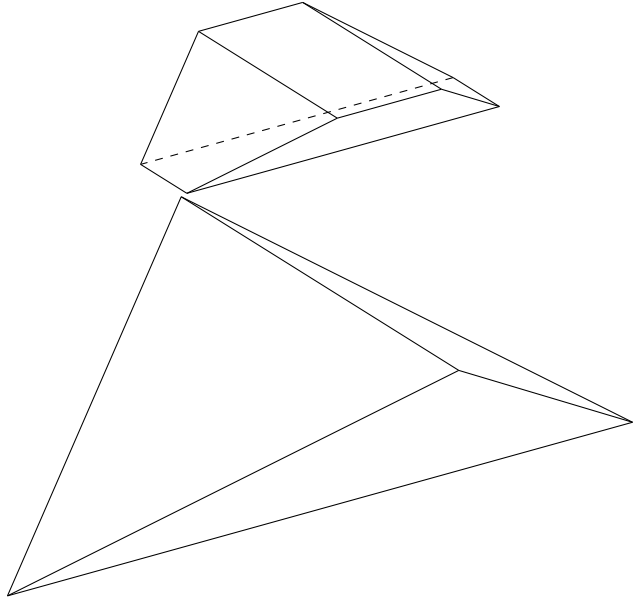
\begin{figure}\label{p1crossp1}
\begin{tikzpicture}[line join=bevel]
\coordinate (A1) at (-2,0,1);
\coordinate (A2) at (-1,0,2);
\coordinate (A3) at (2,0,-1);
\coordinate (A4) at (1,0,-2);
\coordinate (B1) at (-2,1,-1);
\coordinate (B2) at (1,1,2);
\coordinate (B3) at (2,1,1);
\coordinate (B4) at (-1,1,-2);
\draw (A1)--(A2);
\draw (A2)--(A3);
\draw (A3)--(A4);
\draw [dashed](A4)--(A1);
\draw (B1)--(B2);
\draw (B2)--(B3);
\draw (B3)--(B4);
\draw (B4)--(B1);
\draw (A1)--(B1);
\draw (A2)--(B2);
\draw (A3)--(B3);
\draw (A4)--(B4);
\end{tikzpicture}

\begin{tikzpicture}[line join=bevel,]
\coordinate (A1) at (-3,-1,3);
\coordinate (A2) at (3,-1,-3);
\coordinate (B1) at (3,2,3);
\coordinate (B2) at (-3,2,-3);
\draw (A1)--(A2);
\draw (B1)--(B2);
\draw (A1)--(B1);
\draw (A2)--(B2);
\draw (A1)--(B2);
\draw (A2)--(B1);
\end{tikzpicture}
\caption{The construction for $X=\mathbb{P}^1\times\mathbb{P}^1$, $L=(-1,1)$.  The $\mathbb{P}^1$ bundle $\mathbb{P}_X(\mathcal{O}_X(1,-1)\oplus \mathcal{O}_X)$ is contracted to $\mathbb{P}^3$.
}
\label{P3}
\end{figure}
While this example gives some flavor of the general construction, we cannot in general expect the variety $X'$ to be smooth, even when $X$ is.
\end{example}
\section{Singularities}\label{sing}

\begin{theorem}\label{induct}
Let $(X,\Delta)$ be a log Fano pair with $\rho_X>1$ which is $\Q$-factorial.  Then there exists a line bundle $L$ on $X$ and a small compactification $X'$ of the associated $\mathbb{G}_m$ bundle to $L$ with a $\Q$-divisor $\Delta'$ such that the pair $(X',\Delta')$ is also a $\Q$-factorial log Fano variety.
\end{theorem}




Let $X$ be a Mori Dream Space.  The effective cone of $X$ has a decomposition into finitely many Mori chambers \cite{MR1786494}.  Therefore we may choose a Cartier divisor $L$ such that 
\begin{enumerate}
\item  Both $L$ and $-L$ are not effective.
\item  The intersections of the walls of the Mori chamber decomposition of the effective cone with the line segment connecting $-K_X$ and $L$ are transverse, and this segment only intersects one wall at a time, likewise for $-K_X$ and $-L$.

\end{enumerate}
The first condition will ensure that we will construct a small compactification; the second ensures that this compactification will be $\mathbb{Q}$-factorial.

Let $Y$ be the projectivized vector bundle $\mathbb{P}_X(\mathcal{O}_X\oplus \mathcal{O}_X(L))$, and $Y_0$ the complement of $E_0\cup E_\infty$ in $Y$.  Then by \ref{MDS} $Y$ is also a Mori Dream Space.  Hence $Y$ has finitely many small $\Q$-factorial modifications and these correspond to chambers in the effective cone of $Y$.  To understand these chambers, we need to understand divisors on $Y$, and in particular we need an ample divisor on $Y$.

There is the map $\pi:Y\to X$ which exhibits $Y$ as a $\mathbb{P}^1$ bundle over $X$.  Let $A$ be an ample line bundle on $X$. The pullback $\pi^*A$ has positive top intersection with every subvariety of $Y$ which isn't the pullback of a variety on $X$.  The divisors $E_0$  and $E_\infty$ intersect the fibers positively hence by the Nakai-Moishezon criterion \cite[1.2.19]{MR2095471} $\pi^*A+\varepsilon E_0+\varepsilon E_\infty$ is ample for sufficiently small $\varepsilon$.  In the log Fano case we will be taking $A=-(K_X+\Delta)$.

The variety $X$ is normal, so $Y$ is smooth in codimension $2$.  Hence divisors on $Y$ are Cartier in codimension $2$, and we can use the adjuction formula\cite[Rmk 5.47]{MR1658959}.  By adjunction, $K_Y=\pi^*K_X-E_0-E_\infty$.  In the case where $(X,\Delta)$ is log Fano, we will see that there is a divisor $\Gamma$ on $Y$ such that $(Y,\Gamma)$ is log Fano, but even if $X$ is strictly Fano, $Y$ may not be.

Next we will construct $X'$.  By hypothesis, neither $L$ nor $-L$ are effective.  Since $A$ is ample, there are positive rational numbers $a_+$ and $a_-$ such that $A+a_+L$ and $A+a_-L$ lie on the boundary of the effective cone.  Choose $b$ greater than both $a_+$ and $a_-$.  Set $H=A+b(E_0+E_\infty)$. Let $X'=\text{Proj}\bigoplus H^0(Y,\mathcal{O}(nH))$.  There is an induced rational map $f:Y\to X'$ which is regular away from the base locus of $H$.  Since $-pi^*A+\varepsilon E_0$ is ample, the map $f$ is regular and an isomorphism away from the divisors $E_0$ and $E_\infty$.  Thus $X'$ is a compactification of $Y_0$, which is the geometric realization of the $\mathbb{G}_m$ bundle $L$.  We define the divisor $\Delta'$ as the strict transform of $\pi^*\Delta$.

The divisors $E_0$ and $E_\infty$ are the exceptional loci for the rational map $f$.  Since $b$ was larger than both $a_+$ and $a_-$, both divisors are in the stable base locus, and when the base components are removed, the linear series $H$ is not big on either divisor.  Thus the images of $E_0$ and $E_\infty$ are of strictly smaller dimension, so $X'$ is a small compactification of $Y_0$.

\begin{proposition}\label{normality}
$X'$ is normal and $\Q$-factorial.
\end{proposition}

\begin{proof}
It is equivalent to show that the divisor $H$ lies in the interior of a Mori chamber of $Y$\cite[Proposition 1.11]{MR1786494}.

To show this, we will construct a path in $N^1(Y)$ which connects $H$ with the ample divisor $-\pi^*(K_X+\Delta)+\varepsilon E_0$ and intersects the walls of the Mori decomposition in a finite set.  The path will consist of two line segments:

First, connect $H=\pi^*A+bE_0+bE_\infty$ to $\pi^*A+bE_0+\varepsilon E_\infty$ with a line segment.
Then connect $\pi^*A+bE_0$ to $\pi^*A+\varepsilon E_0+\varepsilon E_\infty$ with another line segment.

We will need to analyze what it means geometrically to cross a wall between chambers.  Say we start in the chamber corresponding to the variety $Z$. Each wall of this chamber corresponds to a curve in the nef cone of $Z$.  By \cite[Proposition 1.11]{MR1786494}, there are two things that may happen.  If the curve comes in a family which covers a divisor, that divisor will be contracted, and the result is the variety of the other chamber.

If the curve does not cover a divisor, then contracting the curve yields a variety which is not $\Q$-factorial.  The solution to this difficulty is an operation called a $D$ flip, which is a small modification which creates a new $\Q$-factorial variety, which will be the variety on the other side of the wall.  Specifically, let $g:Z\to W$ be the contracting morphism, and $D$ a $\Q$-Cartier divisor such that $-D$ is $g$-ample.  Then a $D$-flip is a map $g:Z'\to W$ where $Z'$ is a small modification of $Z$, and the strict transform of $D$ in $Z'$ is $\Q$-Cartier and $g'$-ample.  This is always unique when it exists, and will exist when $Z$ is a Mori Dream Space.  See \cite{MR1786494},\cite{MR1658959} for details.

For our purposes, what is important is that when our path hits a wall this corresponds to the divisor becoming trivial on some curve class.  Consider without loss of generality the part of the path where we are adding $E_\infty$.  Then  up until $E_\infty$ is contracted, the curve in question must be contained in the strict transform of $E_\infty$.  Once $E_\infty$ is contracted, $E_\infty$ is a component in the base locus so increasing the coefficient of $E_\infty$ will not change the rational map, and so we will not cross anymore walls.

Now the strict transform of $E_\infty$ is the image of $E_\infty$ under the induced rational map, and this is the birational model of $X$ given by $\text{Proj}\bigoplus_n H^0(X,\mathcal{O}(nA-naL))$.  Since $L$ was chosen so that as $a$ ranges from $0$ to $a_-$ only one curve becomes negative with respect to $A-aL$ at a time, the same is true along the path in $N^1(Y)$.

Since the path began in the interior of a chamber and intersects chamber walls in only finitely many places, it must end in the interior of a chamber.
\end{proof}

For the rest of this section we will consider only the log Fano case:

\begin{proposition}\label{positivity}
When $(X,\Delta)$ is log Fano, the divisor $-(K_{X'}+\Delta')$ is ample on $X'$.
\end{proposition}
\begin{proof}
Since $E_0$ and $E_\infty$ are contracted by $f$, the strict transform of $-\pi^*K_X-\pi^*\Delta$ is equivalent to that of $-\pi^*K_X-\pi^*\Delta+a_+E_0+a_- E_\infty$ on $X'$ and hence ample.  By adjunction, $\pi_0^*K_X$ is the canonical divisor on $Y_0$  (where $\pi_0$ is the restriction of $\pi$ to $Y_0$).  But $X'$ is a small compactification of $Y_0$, so the closure of $\pi_0^*K_X$ in $X'$ is $K_{X'}$, and this is the strict transform of $\pi^*K_X$.  Likewise by definition $\Delta'$ is the strict transform of $\pi^*\Delta$.  Hence $-(K_{X'}+\Delta')$ is ample.
\end{proof}

We will soon show that the pair $(X',\Delta')$ has log terminal singularities.  We know that $X'$ is related to the variety $Y$ by a sequence of $D$-flips and divisorial contractions.  The contracted curves in every case are actually $(K_Y+\Gamma)$ negative, once we know that $(Y,\Gamma)$ is klt, the resulting log flips and contractions will also be klt.

\begin{lemma}\label{klt}
For $0<\varepsilon<1$ the pair $(Y,\pi^*\Delta+(1-\varepsilon)E_0+(1-\varepsilon)E_\infty)$ is klt.
\end{lemma}

\begin{proof}
Set $\Gamma=\pi^*\Delta+(1-\varepsilon)E_0+(1-\varepsilon)E_\infty$.  Recall that $Y=\mathbb{P}_X(\mathcal{O}_X\oplus\mathcal{O}_X(L))$.  Let $\mu:Z\to X$ be a log resolution of the pair $(X,\Delta)$.  Then we can construct a resolution of singularities of $Y$ by taking $W=\mathbb{P}_Z(\mathcal{O}_Z\oplus\mathcal{O}_Z(\mu^*L))$.  We claim the map $\nu: W\to Y$ is a log resolution of the pair $(Y,\Gamma)$.  Let $F_i$ be the exceptional divisors of $\mu$.  The exceptional divisors of $W$ are the pullbacks of those in $Z$ by the projection $\phi$.  Thus, they along with the strict transforms of the boundary divisors $\tilde{E}_0$ and $\tilde{E}_\infty$ intersect transversely.  The strict transforms $\tilde{E}_0$ and $\tilde{E}_\infty$ are smooth since they are each isomorphic to $Z$.
\[
\xymatrix{W\ar[r]^-{\nu}\ar[d]_-{\phi} & Y\ar[d]^-{\pi}
\\ Z\ar[r]^-{\mu} & X
}
\]
Now we must compare discrepancies.  The pair $(X,\Delta)$ is klt, so the log discrepancies $a_i$ are postive:
\[
K_Z+\sum F_i = \mu^*(K_X+\Delta)+a_iF_i
\]
\[
K_W=\phi^*K_Z-\tilde{E}_0-\tilde{E}_\infty
\]
\[
K_W+\sum\phi^*F_i+\tilde{E}_0+\tilde{E}_\infty=\phi^*\mu^*(K_X+\Delta)+a_i\phi^*F_i=\nu^*\pi^*(K_X+\Delta)+a_i\phi^*F_i
\]

But $\pi^*(K_X+\Delta)=K_Y+\Gamma+\varepsilon(E_0+E_\infty)$, so
\[
K_W+\sum\phi^*F_i+\tilde{E}_0+\tilde{E}_\infty=\nu^*(K_Y+\Gamma)+\varepsilon \nu^*E_0+\varepsilon \nu^*E_\infty+a_i\phi^*F_i
\]

The pullbacks $\nu^*E_0$ and $\nu^*E_\infty$ are effective, so all the log discrepancies are positive, and $(Y,\Gamma)$ is klt.
\end{proof}

Thus while $Y$ may not be strictly Fano, for sufficiently small positive $\varepsilon$ the pair $(Y,\Gamma)$ is log Fano, where $\Gamma=\pi^*\Delta+(1-\varepsilon)E_0+(1-\varepsilon)E_\infty$. Now we consider the singularities of $X'$.

\begin{proposition}\label{fanosing}
When $(X,\Delta)$ is log Fano the pair $(X',\Delta')$ is klt.
\end{proposition}

\begin{proof}

We have seen already that $X'$ is normal and $\Q$-factorial, hence $K_{X'}+\Delta'$ is $\Q$-Cartier.  By lemma \ref{klt} the pair $(Y,\Gamma)$ is klt. 

Since $Y$ is a Mori Dream Space, the rational map $Y\dashrightarrow X'$ factors into a series of $D$-flips and divisorial contractions $Y=Z_0\dashrightarrow Z_1 \ldots \dashrightarrow Z_n= X'$ by moving along a path in $N^1(Y)$.  Since $X'$ is $\Q$-factorial, we can choose that path to be a straight line connecting $-(K_Y+\Gamma)$ to $-(K_Y+\Gamma)+F$, where $F$ is effective and satisfies $X'=\text{Proj}\bigoplus H^0(n(-(K_Y+\Gamma)+F))$.

 Let $\Gamma_i$ be the divisor given by the closure of the strict transform of $\Gamma$ in $Z_i$.  Then $\Gamma=\Gamma_0$, and $\Gamma_n=\Delta'$.  Moreover, in the following diagram, we have that $f_{i*}\Gamma_i=g_{i+1*}\Gamma_{i+1}$.
\[
\xymatrix{Z_i \ar[dr]_{f_i}\ar@{-->}^{\psi}[rr] & & Z_{i+1}\ar[dl]^{g_{i+1}}
\\& W & &
}
\]
Assume $(Z_i, \Gamma_i)$ is klt.  There is a single curve class $C$ contracted by $f_i$, and this curve must have negative intersection with the strict transform of $F$.  Thus $C$ is positive on $-(K_{Z_i}+\Gamma_i)$, as this is the strict transform of $-(K_Y+\Gamma)$, and therefore $\psi$ is either a $K_{Z_i}+\Gamma_i$ flip or divisorial contraction of an extremal curve.  In either case, the pair $(Z_{i+1}, \Gamma_i)$ is klt \cite[Cor 3.42, Cor 3.43]{MR1658959}.

Thus $(X', \Delta')$ is klt.
\end{proof}

\begin{proof}[Proof of Theorem \ref{induct}]
Construct $(X',\Delta')$ as above.  Then by \ref{normality} $X'$ is normal and $\Q$-factorial.  By \ref{positivity} $-(K_{X'}+\Delta')$ is ample, and by \ref{fanosing} the pair $(X',\Delta')$ is klt, hence $(X',\Delta')$ is a log Fano pair.
\end{proof}

\begin{example}
Let $X$ be the blowup of $\mathbb{P}^2$ at one point.  This is the Hirzebruch surface $\mathbb{F}_1$, a toric Fano variety.  The Picard group of $X$ is a free abelian group generated by $H$ and $E$ where $H$ is the pullback of the class of a hyperplane in $\mathbb{P}^2$, and $E$ is the exceptional divisor of the blowup.  In this case $K_X=-3H+E_1$, and one can check that $-K_X$ is ample.  We will take $2E-H$ as our divisor $L$.  In Figure \ref{F1} we show the birational transformations from $Y=\mathbb{P}(\mathcal{O}_X\oplus\mathcal{O}_X(L))$ to $X'$.  Note that the first of these is a small modification, a $D$-flip of a curve in one of the exceptional divisors.
\begin{figure}

\begin{tikzpicture}[line join=bevel]
\coordinate (A1) at (-3,0,1);
\coordinate (A2) at (3,0,1);
\coordinate (A3) at (4,0,0);
\coordinate (A4) at (-4,0,0);
\coordinate (B1) at (-1.5,1,1.5);
\coordinate (B2) at (1.5,1,1.5);
\coordinate (B3) at (3,1,0);
\coordinate (B4) at (-3,1,0);
\draw (A1)--(A2);
\draw (A2)--(A3);
\draw[dashed] (A3)--(A4);
\draw (A4)--(A1);
\draw (B1)--(B2);
\draw (B2)--(B3);
\draw (B3)--(B4);
\draw (B4)--(B1);
\draw (A1)--(B1);
\draw (A2)--(B2);
\draw (A3)--(B3);
\draw (A4)--(B4);
\end{tikzpicture}

\begin{tikzpicture}[line join=bevel]
\coordinate (A1) at (-3,0,1);
\coordinate (A2) at (3,0,1);
\coordinate (A3) at (4,0,0);
\coordinate (A4) at (-4,0,0);
\coordinate (B1) at (0,2,2);
\coordinate (C2) at (0,3,1);
\coordinate (C3) at (1,3,0);
\coordinate (C4) at (-1,3,0);
\draw (A1)--(A2);
\draw (A2)--(A3);
\draw[dashed] (A3)--(A4);
\draw (A4)--(A1);
\draw (A1)--(B1);
\draw (A2)--(B1);
\draw (B1)--(C2);
\draw (A3)--(C3);
\draw (A4)--(C4);
\draw (C2)--(C3);
\draw (C2)--(C4);
\draw (C3)--(C4);
\end{tikzpicture}

\begin{tikzpicture}[line join=bevel]
\coordinate (A1) at (-3,0,1);
\coordinate (A2) at (3,0,1);
\coordinate (A3) at (4,0,0);
\coordinate (A4) at (-4,0,0);
\coordinate (B1) at (0,2,2);
\coordinate (C2) at (0,4,0);
\draw (A1)--(A2);
\draw (A2)--(A3);
\draw[dashed] (A3)--(A4);
\draw (A4)--(A1);
\draw (A1)--(B1);
\draw (A2)--(B1);
\draw (B1)--(C2);
\draw (A3)--(C2);
\draw (A4)--(C2);
\end{tikzpicture}

\begin{tikzpicture}[line join=bevel]
\coordinate (A1) at (-5,-1,0);
\coordinate (A2) at (5,-1,0);
\coordinate (B1) at (0,2,2);
\coordinate (C2) at (0,4,0);
\draw (A1)--(A2);
\draw (A1)--(B1);
\draw (A2)--(B1);
\draw (B1)--(C2);
\draw (A1)--(C2);
\draw (A2)--(C2);
\end{tikzpicture}
\caption{A toric example of the construction of $X'$ from $X$.  Here $X$ is the Hirzebruch surface $\mathbb{F}_1$.  We start with a compactified $\mathbb{G}_m$ bundle over $X$, and gradually make birational modifications until both exceptional divisors are contracted.}
\label{F1}
\end{figure}
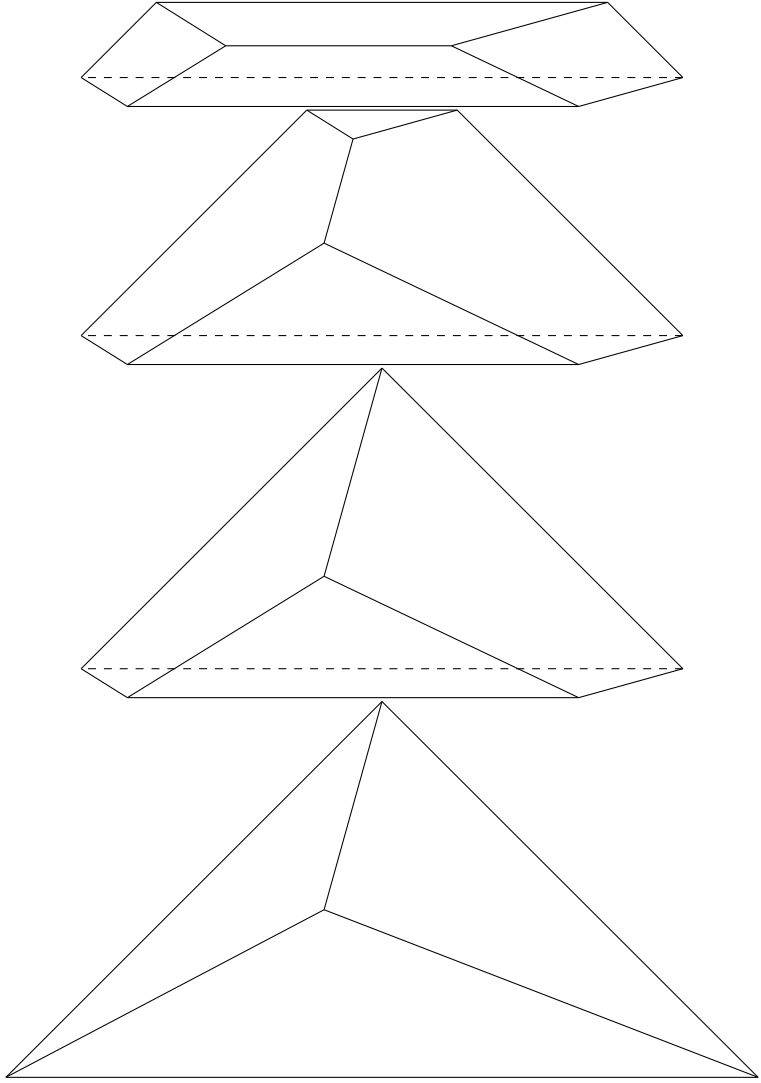
\end{example}

\section{The log Calabi-Yau case}\label{CY}

The theorems and proofs of this section are much the same as for section \ref{sing}, but for the case where the pair $(X,\Delta)$ is log Calabi-Yau.  

\begin{theorem}\label{inductcy}
Let $(X,\Delta)$ be a projective log Calabi-Yau pair with $\rho_X>1$ which is $\Q$-factorial.  Then there exists a line bundle $L$ on $X$ and a small compactification $X'$ of the associated $\mathbb{G}_m$ bundle to $L$ with a $\Q$-divisor $\Delta'$ such that the pair $(X',\Delta')$ is also a projective $\Q$-factorial log Calabi-Yau pair.
\end{theorem}

The construction of $(X',\Delta')$ proceeds as in section \ref{sing}.  As above take $L$ generic such that neither $L$ nor $-L$ is effective, and set $Y=\mathbb{P}_X(\mathcal{O}_X\oplus \mathcal{O}_X(L))$, with $\Gamma=\pi^*\Delta'+E_0+E_\infty$.  The variety $X'$ is obtained by contracting the boundaries $E_0$ and $E_\infty$, and $\Delta'$ is the strict transform of $\pi^*\Delta$.

\begin{lemma}\label{projective}
$X'$ is projective.
\end{lemma}

\begin{proof}
By construction, $X'$ is $\text{Proj}$ of a ring of sections of a line bundle on $Y$, hence $X'$ is projective.
\end{proof}

\begin{lemma}\label{lc}
The pair $(Y,\Gamma)$ is log Calabi-Yau.
\end{lemma}

\begin{proof}
By adjunction, $K_Y=\pi^*K_X-E_0-E_\infty$, so $K_Y+\Gamma$ is trivial.  Thus it remains to show that the pair $(Y,\Gamma)$ is lc.  
Let $\mu:Z\to X$ be a log resolution of the pair $(X,\Delta)$.  Then we can construct a resolution of singularities of $Y$ by taking $W=\mathbb{P}_Z(\mathcal{O}_Z\oplus\mathcal{O}_Z(\mu^*L))$.  We claim the map $\nu: W\to Y$ is a log resolution of the pair $(Y,\Gamma)$.  Let $F_i$ be the exceptional divisors of $\mu$.  The exceptional divisors of $W$ are the pullbacks of those in $Z$ by the projection $\phi$.  Thus, they along with the strict transforms of the boundary divisors $\tilde{E}_0$ and $\tilde{E}_\infty$ intersect transversely.  The strict transforms $\tilde{E}_0$ and $\tilde{E}_\infty$ are smooth since they are each isomorphic to $Z$.
\[
\xymatrix{W\ar[r]^-{\nu}\ar[d]_-{\phi} & Y\ar[d]^-{\pi}
\\ Z\ar[r]^-{\mu} & X
}
\]
Now we must compare discrepancies.  The pair $(X,\Delta)$ is lc, so the log discrepancies $a_i$ are nonnegative:
\[
K_Z+\sum F_i = \mu^*(K_X+\Delta)+a_iF_i
\]
\[
K_W=\phi^*K_Z-\tilde{E}_0-\tilde{E}_\infty
\]
\[
K_W+\sum\phi^*F_i+\tilde{E}_0+\tilde{E}_\infty=\phi^*\mu^*(K_X+\Delta)+a_i\phi^*F_i=\nu^*\pi^*(K_X+\Delta)+a_i\phi^*F_i
\]

But $\pi^*(K_X+\Delta)=K_Y+\Gamma$, so
\[
K_W+\sum\phi^*F_i+\tilde{E}_0+\tilde{E}_\infty=\nu^*(K_Y+\Gamma)+a_i\phi^*F_i
\]

All the log discrepancies are nonnegative, and so $(Y,\Gamma)$ is klt.
\end{proof}

\begin{proposition}\label{lcy}
The pair $(X',\Delta')$ is log Calabi-Yau.
\end{proposition}

\begin{proof}
The boundary of $Y_0$ has codimension $2$ in $X'$, and $K_{X'}+\Delta'$ is trivial on the open locus, so it must be trivial on all of $X'$.  

Let $\tilde{X}$ be a common log resolution of $(X',\Delta')$ and $(Y,\Gamma)$, with exceptional divisors $F_i$.
\[
\xymatrix{& {\tilde{X}}\ar[dr]^-{\beta}\ar[dl]_-{\alpha}
\\ X'& &Y\ar@{-->}[ll]_-{f}
}
\]

Both $K_Y+\Gamma$ and $K_{X'}+\Delta'$ are trivial, so their pullbacks are also trivial.
\[
K_{\tilde{X}}+\sum F_i = \alpha^*(K_Y+\Gamma)+\sum a_i F_i = \sum a_i F_i = \beta^*(K_{X'}+\Delta')+\sum a_i F_i
\]
The discrepancies are the same, and $(Y,\Gamma)$ is lc, so $(X',\Delta')$ is also lc.
\end{proof}

\begin{proof}[Proof of Thm \ref{inductcy}]
Follows from \ref{normality}, \ref{projective}, and \ref{lcy}
\end{proof}
\section{Abelian Covers}\label{final}
So far we have been able to construct new Fano varieties with a smaller Picard number at the cost of replacing the Cox ring with a cyclic quotient.  To recover the original ring we will have to repeatedly take certain cyclic covers, and show that the singularities do not get any worse.  Fortunately, by \cite[Prop 5.20 (iv)]{MR1658959}, if $X'\to X$ is a morphism of normal varieties which is {\'e}tale in codimension $2$, then $X'$ is log terminal iff $X$ is.

In the case where the class group of $X$ is a free abelian group, $\text{Cox}(X)$ is a UFD and hence integrally closed \cite[Prop 4.10]{MR1322960}.  In general there may be torsion in the class group however.
\begin{proposition}\label{normal}
Let $X$ be a normal scheme, and let $D_1 \ldots D_r$ generate the torsion free part of the class group of $X$.  Then $R=\text{Cox}(X; D_1,\ldots, D_r)$ is normal.
\end{proposition}

\begin{proof}
Since $R$ does not depend on the choice of generators, merely the subgroup of $\text{Cl}(X)$ we may assume for simplicity the $D_i$ are irreducible and effective.
The Cox ring is a subring of $K(X)[t_i^{\pm}]$, which is integrally closed.  Let $S$ be the integral closure of $R$.  Then $S\subset K(X)[t_i^{\pm}]$.  Now, let $z$ be an element of $S$.  We have that $z^n=a_{n-1}z^{n-1}+\ldots a_0$ where the $a_i\in R$.  We wish to show that $z\in R$.  Expand $z$ as a Laurent polynomial $\sum (\lambda \prod t_i^{\alpha_i})$.  We must show for each term that $\lambda$ has at worst poles of the prescribed orders along the $D_i$.

Consider an irreducible effective Weil divisor $D$ on $X$.  Then $D$ induces a valuation $v$ on $K(X)$.  This valuation may be extended to $v'$ on $K(X)(t_i)$ as follows.  On a Laurent monomial $v'(f(X)t_i^{\alpha_i}=v(f(X))+\alpha_i$ if $D$ is one of the $D_i$, otherwise $v'(f(X)t_i^{\alpha_i}=v(f(X))$.  On a Laurent polynomial take the minimum of the valuations of each term.  

Now we need to confirm that the new valuation satisfies $v'(\lambda\mu)=v'(\lambda)+v'(\mu)$ for Laurent polynomials $\lambda,\mu$.  Let $\lambda'$ and $\mu'$ consist of the terms of $\lambda$ and $\mu$ which attain the lowest value of $v'$ respectively.  Then since $\lambda'\mu'\neq 0$, The lowest term of $\lambda\mu$ has valuation $v'(\lambda)+v'(\mu)$

On a rational function, take the difference of the valuation on the numerator and denominator.   
Now, by definition of $R$, if $v'(z)\geq 0$ for every $D$ on $X$, $z\in R$.  Apply $v'$ to both sides of $z^n=a_{n-1}z^{n-1}+\ldots a_0$.  Assuming $z$ is nonzero, we see that 

\[
nv'(z)\geq \text{min}_{0\leq i<n}(v'(a_i)+iv'(z))\geq min_{0\leq i<n}(iv'(z))
\]
Thus $v'(z)\geq 0$, so $z\in R$.

\end{proof}

The next step is to show that the covers described by $\ref{cover}$ are {\'e}tale in codimension $2$.  This is essentially a version of Reid's cyclic covering trick \cite[3.6]{MR927963}.    
\begin{lemma}\label{etale}
Assume $X$ is a Mori Dream Space.  Let $L_1 \ldots L_{\rho}$ be Cartier divisors which form a vector space basis for $\text{Pic}_{\mathbb{Q}}(X)$.  Take $R=\bigoplus H^0(X,\mathcal{O}(\sum a_1L_1 \ldots a_{\rho}L_{\rho}))$.  Let $D_i$ be a finite set of Weil divisors such that the subgroup of $\text{Cl}(X)$ generated by the $L_i$ is contained in the subgroup generated by the $D_i$, and such that this subgroup is torsion free.  Set $S=\bigoplus H^0(X,\mathcal{O}(\sum a_iD_i))$.  Then $S$ is a finite extension of $R$, and this extension is {\'e}tale in codimension $2$.
\end{lemma}

\begin{proof}
The fact that the extension is finite follows since $R$ is Noetherian and consists of the invariants of $S$ under the action of an abelian group.

We will find an open set $U\subset \text{Spec}R$ which is {\'e}tale.  Now, we can recover $X$ from $\text{Spec}R$ via a GIT quotient, as in \cite{MR1786494}.  By \cite[Lemma 2.7]{MR1786494}, the unstable locus has codimension $2$, and the remaining open set is a $\mathbb{G}_m^\rho$ bundle over $X$.  Since $X$ is normal, the singular locus of $X$ has codimension at least $2$, so we will take $U$ to be the preimage of the smooth locus of $X$ in the $\mathbb{G}_m^\rho$ bundle.  We must show that for any $u \in U$, the covering is {\'e}tale in a neighborhood of $u$.

Now, let $G$ be the group generated by the $D_i$ modulo the group generated by the $L_i$.  Since $X$ is $\mathbb{Q}$-factorial, $G$ is a finite abelian group, and so $G\cong \mathbb{Z}/n_1\oplus \ldots \oplus \mathbb{Z}/n_k$.  Set $g_1 \ldots g_k$ as a set of generators, and choose $F_1\ldots F_k$ Weil divisors which represent these.  

Let $b_{gi}$ be a set of generators for $S$ as an $R$ module, indexed so that $b_{gi}$ has the class $g$ in $G$ as a divisor.  We can think of each $b_{gi}$ as a section of $L_{gi}+\sum g_jF_j$, where $L_{gi}\in \langle L_1 \ldots L_{\rho} \rangle$.  Now, choose ample divisors $A_1 \ldots A_k$ on $X$ such that $-L_{gi}+\sum g_jA_j$ is a base point free divisor on the nonsingular locus of $X$ for each generator and so that $A_j+F_j$ is base point free.

Now, choose $a_{gj}$ sections of $-L_{gi}+\sum g_jA_j$ which do not pass through $u$, and $x_j$ sections of $A_j+F_j$ which don't pass through $u$.  In a neighborhood of $u$, we can assume the $a_{gj}$ and the $x_j^{n_j}$ are all units.
The element $b_{gj}a_{gj}\prod x_j^{n_j-g_j}$ is in $H^0(X,\sum (n_iA_i+n_iF_i))$, which is in $R$.  Thus $b_{gj}$ is in the module generated by the $x_i$.  So locally the extension is given by adjoining the $x_i$, and for each $x_i$, $x_i^{n_i}$ is a unit.  Thus the extension is {\'e}tale at $u$.
\end{proof}

\begin{proof}[Proof of Thm \ref{main}, \ref{maincy}]
The proof is by induction on the Picard rank $\rho_X$.  In the case $\rho=1$, by Theorem \ref{basecase}, the ring $\bigoplus H^0(X,\mathcal{O}(nL))$ has log terminal singularities.  Let $D$ be a generator of the class group of $X$.  The ring $\text{Cox}(X,D)=\bigoplus H^0(X,\mathcal{O}(nD))$ is normal and so by Lemma \ref{etale} also has log terminal singularities.

Now, assume that Theorem \ref{main} is true for varieties with $\rho=n$.  Given $(X,\Delta)$ log Fano $\mathbb{Q}$-factorial and log terminal with $\rho_X=n+1$, by Theorem \ref{induct}, there is a $\mathbb{G}_m$ bundle $Y_0 \subset \mathbb{P}_X(\mathcal{O}_X\oplus\mathcal{O}_X(L))$on $X$ with a compactification $X'$ which has $\rho_{X'}=n$, such that $X'$ is Fano and $\Q$-factorial and log terminal.  Choose $D_1\ldots D_r$ generators of the torsion free part of $\text{Cl}(X)$, where $D_r$ is a multiple of $L$.  We abuse notation to write $\pi_0^*D_i$ for the image of $\pi_0^*D_i$ under the isomorphism $\text{Cl}(X')\cong\text{Cl}(Y)$. By the inductive hypothesis, $\text{Cox}(X';\pi_0^*D_1\ldots\pi_0^*D_{r-1})$ has log terminal singularities.  By \ref{cover}, $\text{Cox}(X';\pi_0^*D_1\ldots\pi_0^*D_{r-1})$ consists of the invariants of $\text{Cox}(X;D_1\ldots D_r)$ under a cyclic group, so $\text{Cox}(X;D_1\ldots D_r)$ is a cyclic cover of $\text{Cox}(X';\pi_0^*D_1\ldots\pi_0^*D_{r-1})$.  There is a choice of $\mathbb{Q}$-basis $L_1 \ldots L_{n+1}$ for the Picard group of $X$ such that both rings are abelian covers of $H^0(X,\sum a_i L_i)$, and so by \ref{etale} these covers are {\'e}tale in codimension $2$.  Hence the map from $\text{Cox}(X';\pi_0^*D_1\ldots\pi_0^*D_{r-1})$ to $\text{Cox}(X;D_1\ldots D_r)$ is {\'e}tale in codimension $2$. Since $\text{Cox}(X;D_1\ldots D_r)$ is normal, the singularities of $\text{Cox}(X)$ are log terminal.

The log Calabi-Yau case is the same but uses $\ref{inductcy}$ instead of $\ref{induct}$.
\end{proof}
\bibliographystyle{amsalpha}      
\bibliography{algebraicgeometry}
\end{document}